\documentclass[a4paper,11pt]{my_arXiv}

\usepackage{url}

\begin{document}

\title{Diamond-alpha Integral Inequalities on Time Scales\thanks{Accepted to the 8th Portuguese Conference on Automatic Control -- CONTROLO'2008, 21 to 23 July 2008, UTAD University, Vila Real, Portugal.}}

\author{\textbf{Rui A. C. Ferreira$^1$, Moulay Rchid Sidi Ammi$^2$ and Delfim F. M. Torres$^3$}}

\date{Control Theory Group (\textsf{cotg})\\
Centre for Research on Optimization and Control\\
Department of Mathematics, University of Aveiro\\
3810-193 Aveiro, Portugal\\[0.3cm]
$^1$\url{ruiacferreira@ua.pt}, 
$^2$\url{sidiammi@ua.pt},
$^3$\url{delfim@ua.pt}}

\maketitle

% ------------------------------------------------

\begin{abstract}
\noindent \emph{The theory of the calculus of variations was recently extended to the more general time scales setting, 
both for delta and nabla integrals. The primary purpose 
of this paper is to further extend
the theory on time scales, by establishing some basic
diamond-alpha dynamic integral inequalities. We prove generalized
versions of H\"{o}lder, Cauchy-Schwarz, Minkowski, and Jensen's
inequalities. For the particular case when alpha is equal to one
or alpha is equal to zero one gets, respectively, correspondent
delta and nabla inequalities. If we further restrict ourselves by
fixing the time scale to the real or integer numbers, then we
obtain the classical inequalities, whose role in optimal control
is well known. By analogy, we trust that the diamond-alpha
integral inequalities we prove here will be important in the study of control systems on times scales.
}
\medskip

\noindent\textbf{Keywords:} time scales, diamond alpha derivatives, diamond alpha integrals,
inequalities, optimal control.

\medskip

\noindent\textbf{2000 Mathematics Subject Classification:}
26D15, 39A12, 49K05.
\end{abstract}

% ------------------------------------------------

\section{Introduction}

The study of inequalities is an important subject in optimal
control. To give a very simple motivational example, let us
consider the basic problem of the calculus of variations in the
classical setting: given a $C^1$ smooth Lagrangian
$(t,x,v)\rightarrow L(t,x,v)$, determine a minimizer of the
integral functional
\begin{equation*}
J[x(\cdot)] = \int_a^b L\left(t,x(t),\dot{x}(t)\right) dt
\end{equation*}
among all functions
\begin{equation}
\label{eq:adm}
\begin{gathered}
x(\cdot) \in C^1\left([a,b];\mathbb{R}\right) \, , \\
x(a) = \alpha \, , \quad x(b) = \beta \, .
\end{gathered}
\end{equation}
A function $x(\cdot)$ satisfying \eqref{eq:adm} is said to be
\emph{admissible}. The classical theory of the calculus of
variations asserts that if $\tilde{x}(\cdot)$ is an admissible
function satisfying $J[\tilde{x}(\cdot)] \le J[x(\cdot)]$ for all
admissible $x(\cdot)$, then $\tilde{x}(\cdot)$ satisfy the
Euler-Lagrange equation, \textrm{i.e.}
\begin{equation}
\label{eq:EL}
\frac{d}{dt} \left[ \frac{\partial L}{\partial v}\left(t,
\tilde{x}(t), \dot{\tilde{x}}(t)\right) \right]
= \frac{\partial L}{\partial x}\left(t,
\tilde{x}(t), \dot{\tilde{x}}(t)\right) \, .
\end{equation}
Let us fix $L = v^2$, $a = \alpha = 0$, $b = \beta = 1$. Then, the
Euler-Lagrange equation \eqref{eq:EL} give us a unique candidate
to minimizer: $\tilde{x}(t) = t$. It follows from the Schwarz
inequality,
\begin{equation*}
\left(\int_a^b f(t) g(t) dt \right)^2 \le \left(\int_a^b f^2(t) dt
\right) \left(\int_a^b g^2(t) dt \right) \, ,
\end{equation*}
that
\begin{equation*}
1 = x(1) - x(0) = \left(\int_0^1 \dot{x}(t) dt\right)^2 
\le \int_0^1 1 dt \int_0^1 \dot{x}^2(t) dt = J[x(\cdot)] \, .
\end{equation*}
Since $J[\tilde{x}(\cdot)] = 1$, we have just proved that
$\tilde{x}(t) = t$ is indeed the global minimizer of the problem
\begin{gather*}
J[x(\cdot)]=\int_0^1 \dot{x}^2(t) dt \longrightarrow \min \, ,\\
x(\cdot) \in C^1\left([0,1]; \mathbb{R}\right) \, , \\
x(0) = 0 \, , \quad x(1) = 1 \, .
\end{gather*}
We remark that similar conclusion can be obtained from Cauchy's
inequality
\begin{equation*}
\left(\sum_{i=1}^n a_i b_i\right)^2 \le \sum_{i=1}^n a_i^2
\sum_{i=1}^n b_i^2 \, ,
\end{equation*}
if one consideres the correspondent discrete problem of the
calculus of variations. Another example of the usefulness of
integral inequalities in optimal control arises in connection with
the Lavrentiev phenomenon (\textrm{cf.} \textrm{e.g.}, the
celebrated example of Mani\'{a} which makes use of Jenson's
inequality \cite{bookCesari}). Other applications can be found in
\cite{inequalities,IJPAM,jipam-time-scales1}.

Recently, the theory of the calculus of variations was extended to
general time scales, both for delta and nabla integrals: see
\cite{ZbigDel,Delfim-Rui2,Tangier-with-Rui} and references
therein. In this paper we generalize some inequalities to generic
time scales. In order to cover both delta and nabla integral
inequalities, we consider here the more general framework of
diamond-$\alpha$ integrals \cite{notes:diamond,sQ,diamond}.

% ------------------------------------------------

\section{Basics on diamond-$\alpha$ dynamic derivatives and integrals}
\label{sec2}

A nonempty closed subset of $\mathbb{R}$ is called a \emph{time
scale}, being denoted by $\mathbb{T}$.

The \emph{forward jump operator}
$\sigma:\mathbb{T}\rightarrow\mathbb{T}$ is defined by
$$\sigma(t)=\inf{\{s\in\mathbb{T}:s>t}\},\mbox{ for all $t\in\mathbb{T}$},$$
while the \emph{backward jump operator}
$\rho:\mathbb{T}\rightarrow\mathbb{T}$ is defined by
$$\rho(t)=\sup{\{s\in\mathbb{T}:s<t}\},\mbox{ for all
$t\in\mathbb{T}$},$$ with $\inf\emptyset=\sup\mathbb{T}$
(\textrm{i.e.}, $\sigma(M)=M$ if $\mathbb{T}$ has a maximum $M$)
and $\sup\emptyset=\inf\mathbb{T}$ (\textrm{i.e.}, $\rho(m)=m$ if
$\mathbb{T}$ has a minimum $m$).

A point $t\in\mathbb{T}$ is called \emph{right-dense},
\emph{right-scattered}, \emph{left-dense} or \emph{left-scattered}
if $\sigma(t)=t$, $\sigma(t)>t$, $\rho(t)=t$, or $\rho(t)<t$,
respectively.

The \emph{(forward) graininess function}
$\mu:\mathbb{T}\rightarrow[0,\infty)$ and the \emph{backward
graininess function} $\nu:\mathbb{T}\rightarrow[0,\infty)$ are
defined by
$$\mu(t)=\sigma(t)-t \, , \quad \nu(t)=t-\rho(t) \, , \quad
\mbox{for all $t\in\mathbb{T}$},$$ respectively.

Throughout the text we let $a,b\in\mathbb{T}$ with $a<b$ and
consider the interval $[a,b]$ in $\mathbb{T}$. We then define
$[a,b]^k=[a,b]$ if $b$ is left-dense and $[a,b]^k=[a,\rho(b)]$ if
$b$ is left-scattered. Similarly, we define $[a,b]_k=[a,b]$ if $a$
is right-dense and $[a,b]_k=[\sigma(a),b]$ if $a$ is
right-scattered. We denote $[a,b]^k\cap[a,b]_k$ by $[a,b]^k_k$.

We say that a function $f:[a,b]\rightarrow\mathbb{R}$ is
\emph{delta differentiable} at $t\in[a,b]^k$ if there is a number
$f^{\Delta}(t)$ such that for all $\varepsilon>0$ there exists a
neighborhood $U$ of $t$ (\textrm{i.e.},
$U=(t-\delta,t+\delta)\cap\mathbb{T}$ for some $\delta>0$) such
that
$$|f(\sigma(t))-f(s)-f^{\Delta}(t)(\sigma(t)-s)|
\leq\varepsilon|\sigma(t)-s|$$ for all $s\in U$. We call
$f^{\Delta}(t)$ the \emph{delta derivative} of $f$ at $t$.
Similarly, we say that a function $f:[a,b]\rightarrow\mathbb{R}$
is \emph{nabla differentiable} at $t\in[a,b]_k$ if there is a
number $f^{\nabla}(t)$ such that for all $\varepsilon>0$ there
exists a neighborhood $U$ of $t$ (\textrm{i.e.},
$U=(t-\delta,t+\delta)\cap\mathbb{T}$ for some $\delta>0$) such
that
$$|f(\rho(t))-f(s)-f^{\nabla}(t)(\rho(t)-s)|
\leq\varepsilon|\rho(t)-s|$$ for all $s\in U$. We call
$f^{\nabla}(t)$ the \emph{nabla derivative} of $f$ at $t$.

With the above definitions we define a combined dynamic
derivative:
\begin{definition}
Let $f$ be simultaneously \emph{delta} and {nabla} differentiable.
For $t\in[a,b]^k_k$ we define the diamond-$\alpha$ dynamic
derivative $f^{\Diamond_\alpha}(t)$ by
$$f^{\Diamond_\alpha}(t)=\alpha f^\Delta(t)+(1-\alpha)f^\nabla(t) \, ,
\quad 0\leq\alpha\leq 1.$$
\end{definition}

A function $f:\mathbb{T}\rightarrow\mathbb{R}$ is called
\emph{rd-continuous} if it is continuous at right-dense points and
if its left-sided limit exists at left-dense points. We denote the
set of all rd-continuous functions by C$_{\textrm{rd}}$ or
C$_{\textrm{rd}}[\mathbb{T}]$, and the set of all delta
differentiable functions with rd-continuous derivative by
C$_{\textrm{rd}}^1$ or C$_{\textrm{rd}}^1[\mathbb{T}]$.

It is known that rd-continuous functions possess an
\emph{antiderivative}, \textrm{i.e.}, there exists a function $F$
with $F^\Delta=f$, and in this case the $\Delta$-\emph{integral}
is defined by $\int_{a}^{b}f(t)\Delta t=F(b)-F(a)$. Analogously,
we define the $\nabla$-\emph{integral} \cite{advance}.

Suppose now that the following integrals are well defined:
\begin{equation}
\label{deltaint} \int_a^b f(t)\Delta t
\end{equation}
and
\begin{equation}
\label{nablaint} \int_a^b f(t)\nabla t.
\end{equation}
Then, we introduce the following definition.

\begin{definition}
The $\Diamond_\alpha$ integral of function $f$ is given by
\begin{equation}
\label{diamint} \int_a^b f(t)\Diamond_\alpha t=\alpha\int_a^b
f(t)\Delta t+(1-\alpha)\int_a^b f(t)\nabla t,
\end{equation}
$0\leq\alpha\leq 1$.
\end{definition}
We may notice that when $\alpha=1$, (\ref{diamint}) reduces to the
standard $\Delta$-integral, and when $\alpha=0$ we obtain the
standard $\nabla$-integral.

\begin{remark}
A sufficient condition for the existence of (\ref{deltaint}) and
(\ref{nablaint}) is the continuity of function $f$ on $[a,b]$.
\end{remark}

Next lemma provides some trivial but useful results for what
follows.

\begin{lemma}\label{lem1}
Assume that $f$ and $g$ are continuous functions on $[a,b]$.
\begin{enumerate}
    \item If $f(t)\geq 0$ for all $t\in[a,b]$,
    then $$\int_a^b f(t)\Diamond_\alpha t\geq
    0 .$$
    \item If $f(t)\leq g(t)$ for all $t\in[a,b]$,
    then $$\int_a^b f(t)\Diamond_\alpha t\leq\int_a^b g(t)\Diamond_\alpha
    t .$$
    \item If $f(t)\geq 0$ for all $t\in[a,b]$, then $f(t)=0$ if and
    only if $\int_a^b f(t)\Diamond_\alpha t=0$.
\end{enumerate}
\end{lemma}

\begin{proof}
Let $f(t)$ and $g(t)$ be continuous functions on $[a,b]$.
\begin{enumerate}
    \item Since $f(t)\geq 0$ for all $t\in[a,b]$, we know (see
    \cite{livro,advance}) that $\int_a^b f(t)\Delta t\geq 0$
    and $\int_a^b f(t)\nabla t\geq 0$. Since $\alpha\in[0,1]$, the result follows.
    \item Let $h(t)=g(t)-f(t)$. Then, $\int_a^b h(t)\Diamond_\alpha t\geq
    0$ and the result follows from Theorem~3.7 (i) of
    \cite{diamond}.
    \item If $f(t)=0$ for all $t\in[a,b]$, the result is immediate. Suppose now that
    there exists $t_0\in[a,b]$ such that $f(t_0)>0$. It is easy to
    see that at least one of the integrals $\int_a^b f(t)\Delta
    t$ or $\int_a^b f(t)\nabla t$ is strictly positive. Then,
    we have the contradiction $\int_a^b f(t)\Diamond_\alpha t>0$.
\end{enumerate}
\end{proof}

For completeness, we also prove:

\begin{lemma}
Let $c$ be a real constant. Then,
$$\int_a^b c \Diamond_\alpha
t=c(b-a) .$$
\end{lemma}

\begin{proof}
We have
\begin{equation*}
\begin{split}
\int_a^b 1 \Diamond_\alpha t &=\alpha\int_a^b 1 \Delta
t+(1-\alpha)\int_a^b 1 \nabla t\\
&=\alpha(b-a)+(1-\alpha)(b-a)=b-a ,
\end{split}
\end{equation*}
and the result follows using Theorem~3.7 (ii) of \cite{diamond}.
\end{proof}

% ------------------------------------------------

\section{Main Results}

We start in \S\ref{sub:sec:H:CS:M} by proving H\"{o}lder's
inequality using $\Diamond_\alpha$ integrals. Then, we obtain the
Cauchy-Schwarz and Minkowski inequalities as corollaries. In
\S\ref{sub:sec:J} we prove a Jensen's diamond-$\alpha$ integral
inequality.

% ------------------------------------------------

\subsection{H\"{o}lder, Cauchy-Schwarz and Minkowski Inequalities}
\label{sub:sec:H:CS:M}

\begin{theorem}[H\"{o}lder's Inequality] For continuous
functions $f$, $g:[a,b]\rightarrow\mathbb{R}$, we have:
\begin{equation*}
\int_a^b |f(t)g(t)|\Diamond_\alpha t
\leq \left[\int_a^b |f(t)|^p\Diamond_\alpha
t\right]^{\frac{1}{p}}\left[\int_a^b |g(t)|^q\Diamond_\alpha
t\right]^{\frac{1}{q}},
\end{equation*}
where $p>1$, and $q=\frac{p}{p-1}$.
\end{theorem}

\begin{proof}
For nonnegative real numbers $x$ and $y$, we have (see
\textrm{e.g.} \cite{inesurvey})
\begin{equation}
\label{des1}
x^{\frac{1}{p}}y^{\frac{1}{q}}\leq\frac{x}{p}+\frac{y}{q}.
\end{equation}
If
$$\left[\int_a^b |f(t)|^p\Diamond_\alpha t\right]^{\frac{1}{p}}\left[
\int_a^b |g(t)|^q\Diamond_\alpha t\right]^{\frac{1}{q}}=0,$$ then
$f(t)=0$ or $g(t)=0$ and the result follows by item~3 of
Lemma~\ref{lem1}. So, we assume that
$$\left[\int_a^b |f(t)|^p\Diamond_\alpha
t\right]^{\frac{1}{p}}\left[\int_a^b |g(t)|^q\Diamond_\alpha
t\right]^{\frac{1}{q}}\neq 0$$
and apply (\ref{des1}) to
$$x(t)=\frac{|f(t)|^p}{\int_a^b |f(\tau)|^p\Diamond_\alpha\tau}\
\mbox{and}\ y(t)=\frac{|g(t)|^q}{\int_a^b
|g(\tau)|^q\Diamond_\alpha\tau}$$ obtaining
\begin{equation*}
\frac{|f(t)|}{[\int_a^b
|f(\tau)|^p\Diamond_\alpha\tau]^{\frac{1}{p}}}\frac{|g(t)|}{[
\int_a^b |g(\tau)|^q\Diamond_\alpha\tau]^{\frac{1}{q}}} 
\leq \frac{1}{p}\frac{|f(t)|^p}{\int_a^b
|f(\tau)|^p\Diamond_\alpha\tau}
+\frac{1}{q}\frac{|g(t)|^q}{\int_a^b
|g(\tau)|^q\Diamond_\alpha\tau}.
\end{equation*}
Integrating both sides of the obtained inequality between $a$ and
$b$, and using item~2 of Lemma~\ref{lem1}, we obtain
\begin{align*}
&\int_a^b\frac{|f(t)|}{[\int_a^b
|f(\tau)|^p\Diamond_\alpha\tau]^{\frac{1}{p}}}\frac{|g(t)|}{[\int_a^b
|g(\tau)|^q\Diamond_\alpha\tau]^{\frac{1}{q}}}\Diamond_\alpha
t\\
&\leq\int_a^b\left\{\frac{1}{p}\frac{|f(t)|^p}{\int_a^b
|f(\tau)|^p\Diamond_\alpha\tau}+\frac{1}{q}\frac{|g(t)|^q}{\int_a^b
|g(\tau)|^q\Diamond_\alpha\tau}\right\}\Diamond_\alpha t\\
&=\frac{1}{p}+\frac{1}{q} =1,
\end{align*}
from which we get H\"{o}lder's inequality.
\end{proof}

If we let $p=q=2$ in the above theorem, we obtain the
Cauchy-Schwarz inequality.

\begin{theorem}[Cauchy-Schwarz Inequality] For continuous
functions $f$, $g:[a,b]\rightarrow\mathbb{R}$ we have
\begin{equation*}
\int_a^b |f(t)g(t)|\Diamond_\alpha t
\leq \sqrt{\left[\int_a^b |f(t)|^2\Diamond_\alpha
t\right]\left[\int_a^b |g(t)|^2\Diamond_\alpha t\right]} \, .
\end{equation*}
\end{theorem}

Next, we use H\"{o}lder's inequality to deduce Minkowski's
inequality.

\begin{theorem}[Minkowski's Inequality]
For continuous functions $f$, $g:[a,b]\rightarrow\mathbb{R}$ we
have
\begin{equation*}
\left[\int_a^b |f(t)+g(t)|^p\Diamond_\alpha
t\right]^{\frac{1}{p}}
\leq \left[\int_a^b |f(t)|^p\Diamond_\alpha t\right]^{\frac{1}{p}}
+\left[\int_a^b |g(t)|^p\Diamond_\alpha t\right]^{\frac{1}{p}},
\end{equation*}
where $p>1$.
\end{theorem}

\begin{proof}
First, note that
\begin{align}
\int_a^b |f(t)&+g(t)|^p\Diamond_\alpha t =\int_a^b |f(t)+g(t)|^{p-1}|f(t)+g(t)|\Diamond_\alpha t\nonumber\\
&\leq \int_a^b|f(t)||f(t)+g(t)|^{p-1}\Diamond_\alpha t +\int_a^b|g(t)| \, |f(t)+g(t)|^{p-1}\Diamond_\alpha t ,
\label{mink1}
\end{align}
by the triangle inequality. Next, we apply H\"{o}lder's inequality
with $q=p/(p-1)$ to (\ref{mink1}) to obtain
\begin{multline}
\label{eq:prf:MI}
\int_a^b |f(t)+g(t)|^p\Diamond_\alpha t 
\leq\left[\int_a^b|f(t)|^p\Diamond_\alpha
t\right]^{\frac{1}{p}}\left[\int_a^b
|f(t)+g(t)|^{(p-1)q}\Diamond_\alpha
t\right]^{\frac{1}{q}}\\
+\left[\int_a^b|g(t)|^p\Diamond_\alpha
t\right]^{\frac{1}{p}}\left[\int_a^b
|f(t)+g(t)|^{(p-1)q}\Diamond_\alpha
t\right]^{\frac{1}{q}} \, .
\end{multline}
The right-hand side of \eqref{eq:prf:MI} equals
\begin{equation}
\label{eq:rhs} \left\{\left[\int_a^b|f(t)|^p\Diamond_\alpha
t\right]^{\frac{1}{p}} +\left[\int_a^b|g(t)|^p\Diamond_\alpha
t\right]^{\frac{1}{p}}\right\}
\left[\int_a^b|f(t)+g(t)|^p\Diamond_\alpha
t\right]^{\frac{1}{q}}
\end{equation}
and dividing the left-hand side of \eqref{eq:prf:MI} and
\eqref{eq:rhs} by
$$\left[\int_a^b|f(t)+g(t)|^p\Diamond_\alpha
t\right]^{\frac{1}{q}},$$ we arrive to Minkowski's inequality.
\end{proof}

% ------------------------------------------------

\subsection{Jensen's Inequality}
\label{sub:sec:J}

\begin{theorem}[Jensen's Inequality] Let $c$ and $d$ be real numbers.
Suppose that $g:[a,b]\rightarrow(c,d)$ is continuous and
$F:(c,d)\rightarrow\mathbb{R}$ is convex. Then,
$$F\left(\frac{\int_a^b g(t)\Diamond_\alpha
t}{b-a}\right)\leq\frac{\int_a^b F(g(t))\Diamond_\alpha t}{b-a}.$$
\end{theorem}

\begin{proof}
Let $x_0\in(c,d)$. Then, there exists \cite{inesurvey}
$\gamma\in\mathbb{R}$ such that
\begin{equation}
\label{jen1} F(x)-F(x_0)\geq\gamma(x-x_0)\ \mbox{for all}\
x\in(c,d).
\end{equation}
Since $g$ is continuous, we can put
$$x_0=\frac{\int_a^b g(t)\Diamond_\alpha
t}{b-a}.$$ A convex function defined on an open interval is
continuous, so $F\circ g$ is also continuous, and hence we may
apply (\ref{jen1}) with $x=g(t)$ and integrate from $a$ to $b$ to
obtain
\begin{equation*}
\int_a^b F(g(t))\Diamond_\alpha t -\int_a^b F(x_0)\Diamond_\alpha t \geq \gamma\int_a^b [g(t)-x_0]\Diamond_\alpha t \, ,
\end{equation*}
that is,
\begin{equation*}
\int_a^b F(g(t))\Diamond_\alpha t-(b-a)F\left(\frac{\int_a^b
g(t)\Diamond_\alpha t}{b-a}\right)
\geq\gamma\left[\int_a^b g(t)\Diamond_\alpha t-x_0(b-a)\right]=0
\, .
\end{equation*}
This directly yields the Jensen's inequality.
\end{proof}

Further extensions are in progress and will appear elsewhere.

% ------------------------------------------------

\section{Example}

Let $\mathbb{T}=\mathbb{Z}$ and $n\in\mathbb{N}$. Fix $a=1$,
$b=n+1$ and consider a function
$g:\{1,\ldots,n+1\}\rightarrow(0,\infty)$. Obviously, $F=-\log$ is
convex and continuous on $(0,\infty)$, so we may apply Jensen's
inequality to obtain
\begin{align*}
\log &\left[\frac{1}{n}\left(\alpha\sum_{t=1}^n
g(t)+(1-\alpha)\sum_{t=2}^{n+1}g(t)\right)\right] 
=\log\left[\frac{1}{n}\int_1^{n+1}g(t)\Diamond_\alpha
t\right]\\
&\geq\frac{1}{n}\int_1^{n+1}\log(g(t))\Diamond_\alpha t
=\frac{1}{n}\left[\alpha\sum_{t=1}^n
\log(g(t))+(1-\alpha)\sum_{t=2}^{n+1}\log(g(t))\right]\\
&=\log\left\{\prod_{t=1}^n
g(t)\right\}^{\frac{\alpha}{n}}+\log\left\{\prod_{t=2}^{n+1}
g(t)\right\}^{\frac{1-\alpha}{n}} \, .
\end{align*}
Hence,
\begin{equation*}
\frac{1}{n}\left(\alpha\sum_{t=1}^n
g(t)+(1-\alpha)\sum_{t=2}^{n+1}g(t)\right)
\geq\left\{\prod_{t=1}^n
g(t)\right\}^{\frac{\alpha}{n}}\left\{\prod_{t=2}^{n+1}
g(t)\right\}^{\frac{1-\alpha}{n}} \, .
\end{equation*}
When $\alpha=1$, we obtain the well-known arithmetic-mean
geometric-mean inequality
$$\frac{1}{n}\sum_{t=1}^n
g(t)\geq\left\{\prod_{t=1}^n g(t)\right\}^{\frac{1}{n}}.$$

% ------------------------------------------------

\section*{Acknowledgments}

This work was partially supported by the \emph{Portuguese
Foundation for Science and Technology} (FCT), through the
\emph{Control Theory Group} (cotg) of the R\&D unit CEOC of the
University of Aveiro (\url{http://ceoc.mat.ua.pt}), cofinanced by
the European Community fund FEDER/POCI 2010. The first author also acknowledges the support given by the FCT Ph.D. fellowship
SFRH/BD/39816/2007.

% ------------------------------------------------

% ------------------------------------------------

\end{document}